\documentclass[a4paper,10pt]{article}
\usepackage[english]{babel}
\usepackage[utf8]{inputenc}
\usepackage{amssymb}
\usepackage{amsmath}
\usepackage{amsthm}
\usepackage{a4wide} 
\usepackage[colorlinks=true]{hyperref}

\newcommand{\Tr}{\mathrm{Tr}}
\newtheorem{theorem}{Theorem}[section]
\newtheorem{corollary}[theorem]{Corollary}
\newtheorem{lemma}[theorem]{Lemma}
\newtheorem{remark}{Remark}

\title{A matrix Burkholder-Davis-Gundy inequality}
\author{Tom Maitre} 

\begin{document}

\maketitle 

\begin{abstract} 
We prove an inequality for the spectral norm of matrix valued stochastic integrals. This inequality can be seen either as a non-commutative version of the\break Burkholder–Davis–Gundy inequality or as an extension of the non-commutative Khintchine inequality of Lust-Piquard to stochastic integrals. The proof relies on a version of Freedman's inequality for matrix valued martingales.
\end{abstract}

\section{Introduction and preliminaries}
The non-commutative Khintchine inequality, discovered by Lust-Piquard, asserts that there exists two universal constant $c$ and $C$ such that for a random variable $X$ of the form $\sum_{i=1}^{N}H_{i}\gamma_{i}$ where $\gamma_{i}$ are independent standard Gaussian variables and $H_{i}$ are deterministic symmetric matrices of size $n \times n$, \begin{equation} \label{K} c\left\lVert \sum_{i=1}^{N}H_{i}^{2}\right\rVert^{1/2} \le \mathbb{E}\lVert X\rVert \le C\sqrt{\log n}\left\lVert \sum_{i=1}^{N}H_{i}^{2}\right\rVert^{1/2}\end{equation}
where $\lVert A\rVert$ denotes the spectral norm of the symmetric matrix A. For further details on inequality, we refer to \cite{Lus91} or \cite{Pis03}, and \cite{Ban16} for an improved variant of this inequality. 

The Burkholder-Davis-Gundy's (BDG) inequality asserts that
\begin{equation}\label{BDG} c_{p}\left(\mathbb{E}\langle X \rangle_{t}^{p/2}\right)^{1/p} \le \mathbb{E}\sup_{0\le s\le t}\lvert X_{s}\rvert^{p} \le C_{p}\left(\mathbb{E}\langle X\rangle_{t}^{p/2}\right)^{1/p}\end{equation}
where $X$ is a continuous martingale, $\langle X \rangle$ is its quadratic variation, and $c_{p}, C_{p}$ are constants depending only on $p$. It is an important result in stochastic analysis that has numerous applications. Our main result of this paper generalizes both (\ref{K}) and (\ref{BDG}).

Let $(\Omega,\mathcal{F},\mathbb{P})$ be a probability space equipped with a filtration $(\mathcal{F}_{t})$ and \\$(B_t=(B^1_t, \dotsc,B^N_t))$ be a standard $N$-dimensional Brownian motion.
Throughout this paper, we consider an $(\mathcal{F}_{t})$-local martingale $X$ which can be written, for all $t \ge 0$, as 
\begin{equation} \label{eq_Xt}
X_{t} = \int_{0}^{t}\sum_{i=1}^{N}H_{i,s}dB_{s}^{i}
\end{equation} 
where $(H_{i,t})_{t}$ are progressively measurable processes starting from 0 and taking values in the space $n \times n$ symmetric matrices. For every matrix $A,$ we denote $\Tr(A)$ its trace. 
We assume that $(H_{i,s})_{s}$ satisfies $$\Tr\left(\int_{0}^{t}\sum_{i=1}^{N}H_{i,s}^{2}ds\right) < \infty,\quad \forall t > 0$$ almost surely, which ensures that the stochastic integral~(\ref{eq_Xt}) is well-defined. We define the quadratic variation of $X$,
denoted $\langle X\rangle$, as follows:
\[\langle X\rangle_{t} = \int_{0}^{t}\sum_{i=1}^{N}H_{i,s}^{2}ds , \quad \forall t\geq 0.\]
This process plays a significant role in our main theorem,
which we now state. 
\begin{theorem}\label{thm_main}
Let $(X_{t})$ be a stochastic process of the form~(\ref{eq_Xt}). Then, there exists a universal constant $C$ such that, for all $p \in \mathbb{N}^{*}$ and all $t \in \mathbb{R}_{+}^{*},$
\begin{equation}\label{eq_main} 
\left(\mathbb{E}\sup_{0\le s\le t}\lVert X_{s} \rVert^{p}\right)^{\frac{1}{p}} \le C\sqrt{p+\log n}\;\left(\mathbb{E}\lVert \langle X \rangle_{t}\rVert^{\frac{p}{2}}\right)^{\frac{1}{p}}.
\end{equation}
Our proof gives $C=12\sqrt{2\log(2)}$.
\end{theorem}   

The $\log^{1/2}n$ correction is necessary in general, as we shall see below. We can compare Theorem~\ref{thm_main} with the Burlholder-Davis-Gundy's inequality. In fact, this theorem allows us to control the moments of the supremum of the spectral norm of our process by the spectral norm of its quadratic variation. Thus this extends the BDG inequality, in the case of symmetric matrices. 

Furthermore, if we take $p = 1$, we 
get  
\begin{equation}\label{eq_maincor} 
\mathbb{E}\sup_{0\le s\le t}\lVert X_{s}\rVert \lesssim \sqrt{\log n} \;
\left( \mathbb{E} \left\Vert \int_{0}^{t}\sum_{i=1}^{N}H_{i,s}^{2}ds \right\Vert^{1/2}\right).
\end{equation} 
where $a\lesssim b$ means there exists an universal constant $k$ such that $a \le kb.$ If we consider the special 
case where the matrices $(H_{i,t})$ are deterministic and constant over 
time, applying~(\ref{eq_maincor}) at $t=1$ yields in particular 
\[ 
\mathbb{E} \left\Vert \sum_{i=1} \gamma_i H_i \right\Vert
\lesssim \sqrt{\log n}\left\lVert \sum_{i=1}^{N}H_{i}^{2}\right\rVert^{1/2} , 
\] 
where the variables $\gamma_i$ are independent standard Gaussian variables. 
This is the non-commutative Khintchine inequality of Lust-Piquard, in 
the form put forward by Tropp~\cite{Tro16}. Note that the $\sqrt{\log n}$
factor is in general necessary (see e.g.~\cite{Tro16} for the details), 
which shows that also in~(\ref{eq_main}) the $\sqrt{ \log n}$ is needed. 

\noindent{}Noncommutative versions of the Burkholder–Davis–Gundy inequalities have been established in the context of free probability, notably in \cite{Bia98}, \cite{Jun03} and \cite{Pis97}. In \cite{Pis97}, a Khintchine-type inequality is obtained for sums of matrices with Bernoulli entries. In the continuous setting, Biane and Speicher \cite{Bia98} establish a noncommutative Khintchine-type inequality in the framework of free probability. Specializing this result to our setting yields the following inequality : $$\mathbb{E}\lVert X_{t}\rVert \le 2\sqrt{2}\;\mathbb{E} \left(\int_{0}^{t}\left\lVert\sum_{i=1}^{N}H_{i,s}\right\rVert^{2}\;ds\right)^{1/2}.$$ This result is substantially different from ours (\ref{eq_maincor}). Although the constant is universal, it comes at the cost of removing the supremum over time and introducing cross terms between the processes 
$(H_{i,t})$. The two results are not directly comparable, as neither implies the other.

Another result presented in our paper is an extension of the following inequality, due to Freedman \cite{Fre75}. If $(M_{t})_{t\ge 0}$ is a continuous real-valued local martingale starting from 0, then, for all $u > 0$ and $\sigma \in \mathbb{R},$ $$\mathbb{P}\left(\exists t > 0 , \;M_{t} \ge u\; \; \text{and}\; \; \langle M \rangle_{t} \le \sigma^{2}\right) \le \text{exp}\left(-\frac{u^{2}}{2\sigma^{2}}\right).$$ Tropp extended this result to the matrix-valued case, but only in  discrete time (see \cite{Tro11}). Our version of this inequality reads as follows

\begin{theorem}[Freedman's matrix inequality]\label{freedmann} Let $(X_{t})_{t\ge 0}$ be a stochastic process of the form $(\ref{eq_Xt})$.  Let $\sigma, u \in \mathbb{R}$. Then, \[\mathbb{P}\left(\exists t >0, \lambda_{\text{max}}(X_{t}) \ge u \;\;\text{and} \; \;\left\lVert \langle X \rangle_{t} \right\rVert \le \sigma^{2}\right) \le ne^{-\frac{u^{2}}{2\sigma^{2}}}.\]
\end{theorem}   

To prove Theorem \ref{thm_main}, we use the so-called \textit{good $\lambda$-inequality} method. For applications of such inequalities to Burkholder-Davis-Gundy type results, see e.g. \cite{Ban98} or \cite[Chapter IV]{Rev99}.
In our setting, it takes the form of the following refinement of Theorem \ref{freedmann} :

\begin{theorem}\label{good}
    Let $(X_{t})_{t\ge 0}$ be a stochastic process of the form $(\ref{eq_Xt})$. Let $u, \sigma \in \mathbb{R}$. Then, \[\mathbb{P}\left(\exists t > 0,\; \lambda_{\text{max}}(X_{t}) \ge 2u \;\;\text{and}\;\;\lVert\langle X\rangle_{t}\rVert\le \sigma^{2}\right) \le ne^{-\frac{u^{2}}{2\sigma^{2}}}\mathbb{P}\left(\exists t > 0,\; \lambda_{\text{max}}(X_{t})\ge u\right).\]
\end{theorem}

Using different methods, we also prove an inequality for Schatten norms of matrices of the form (\ref{eq_Xt}). Recall that for all $p \ge 1$, the Schatten $p$-norm of a symmetric matrix $A$ is given by $\lVert A\rVert_{p}^{p} = \sum_{i=1}^{n}\lvert \lambda_{i}\rvert^{p}$ where $\lambda_{1},\dots, \lambda_{n}$ are the eigenvalues of $A$. This indeed defines a norm on the set of symmetric matrices.

\begin{theorem} \label{th_schatten}
Let $(X_{t})_{t\ge 0}$ be a stochastic process of the form given in $(\ref{eq_Xt})$, and $p \in \mathbb{N}^{*}$. Then,
\[\mathbb{E}\lVert X_{t}\rVert_{2p}^{2} \le (2p-1)\left(\mathbb{E}\int_{0}^{t}\left\lVert \sum_{i=1}^{N}H_{i,s}^{2}\right\rVert_{p}ds\right).\]
\end{theorem}   

A classical approach to studying the spectral norm of a matrix is to use the well-known inequalities, valid for all $p \ge 1$ : \begin{equation}\label{schatten}\lVert A \rVert \le \lVert A \rVert_{p} \le n^{1/p}\lVert A \rVert\end{equation} where the upper bound is derived by bounding the trace as $n$ times the maximal eighenvalue. Using Theorem \ref{th_schatten} with $p \approx \log n$, Doob's maximal inequality and (\ref{schatten}), we see that Theorem \ref{th_schatten} yields the following 
\[\mathbb{E}\sup_{0\le s\le t}\lVert X_{s}\rVert \le C\sqrt{\log n}\left(\int_{0}^{t}\mathbb{E}\left\lVert \sum_{i=1}^{N}H_{i,s}^{2}\right\rVert ds\right)^{\frac{1}{2}}.\]
However, this approach yields a less precise inequality than Theorem \ref{thm_main}, since the square root lies outside both the expectation and the integral, and the spectral norm is inside the integral, in contrast with (\ref{eq_maincor}). 

This paper is organized as follows. First, we prove Theorem \ref{th_schatten}. The proof of this theorem uses Itô's lemma. In the second section we first present the proof of Theorem \ref{freedmann}, and then deduce Theorem \ref{good} . Finally, we use the latter to establish our main theorem.

\section{A bound for the Schatten norms}

In this section, we prove Theorem \ref{th_schatten}. Our proof is inspired by Tropp's proof of the matrix Khintchine inequality, which relies on the Gaussian integration by parts formula. In our context, integration by parts is replaced by the use of Itô's formula. For this, we require the following lemma.   

\begin{lemma}\label{lemma1}
Suppose that $H$ and $A$ are Hermitian matrices of the same size. Let $q$ and $r$ be integers that satisfy $0 \le q\le r$. Then, \[\Tr(HA^{q}HA^{r-q}) \le \Tr(H^{2}\lvert A\rvert^{r})\] where $\lvert A\rvert = (A^2)^{1/2}$ denotes the absolute value of $A$.
\end{lemma}   
\noindent{}See e.g. \cite[section 7]{Kla24} for proof of this lemma. We are in a position to prove the main result of this section. 

\begin{proof}[Proof of Theorem \ref{th_schatten}] 
By Itô's lemma, for all sufficiently regular $f$, \begin{equation}\label{Ito}df(X_{t}) = \sum_{i=1}^{N}\langle \nabla f(X_{t}); H_{i,t}\rangle dB_{t}^{i} + \frac{1}{2}\sum_{i=1}^{N}\langle \nabla^{2}f(X_{t})H_{i,t} ; H_{i,t}\rangle dt,\end{equation} where $\nabla^{2}f$ is the Hessian of $f$, $\langle A;B\rangle = \Tr(AB)$ denotes the usual scalar product on the space of symmetric matrices, and $d$ represents the Itô derivative.\\
Now, consider $f(X) = \Tr(X^{2p})^{1/p}$, which is twice differentiable. For all Hermitian matrices $X, H$, we have \[\langle\nabla f(X); H\rangle = 2\Tr(X^{2p})^{\frac{1}{p}-1}\Tr(X^{2p-1}H),\]
and
\begin{align*} \left\langle \nabla^{2} f(X) H; H\right\rangle &= 2\Tr(X^{2p})^{\frac{1}{p}-1}\sum_{k=0}^{2p-2}\Tr(X^{k}HX^{2p-2-k}H) + 4p\left(\frac{1}{p}-1\right)\Tr(X^{2p-1}H)^{2}\Tr(X^{2p})^{\frac{1}{p}-2}\\
&\le 2(2p-1)\Tr(X^{2p})^{\frac{1}{p}-1}\Tr(X^{2p-2}H^{2})\end{align*} 
by Lemma \ref{lemma1}, using that $\lvert X^{2p-2}\rvert = (X^{4p-4})^{1/2} = X$, and nothing that the last term on the right-hand side of the equality is non-positive. Plugging this back into (\ref{Ito}) yields
\begin{equation}\label{avantholder} df(X_{t}) \le 2\Tr(X_{t}^{2p})^{\frac{1}{p}-1}\sum_{i=1}^{N} \Tr(X_{t}^{2p-1}H_{i,t}) dB_{t}^{i} + (2p-1)\lVert X_{t}\rVert_{2p}^{2(1-p)}\left\langle\sum_{i=1}^{N}H_{i,t}^{2};X_{t}^{2p-2}\right\rangle dt.\end{equation}
Applying Hölder's inequality for the Schatten norms (see e.g. \cite{Bar05}), we get
\[\left\langle \sum_{i=1}^{N}H_{i,t}^{2};X_{t}^{2p-2}\right\rangle \le \left\lVert X_{t}\right\rVert_{2p}^{2(p-1)}\left\lVert\sum_{i=1}^{N}H_{i,t}^{2}\right\rVert_{p}.\]
Plugging this back in (\ref{avantholder}), we obtain
\[ df(X_{t}) \le 2\Tr(X_{t}^{2p})^{\frac{1}{p}-1}\sum_{i=1}^{N} \Tr(X_{t}^{2p-1}H_{i,t}) dB_{t}^{i} + (2p-1)\left\lVert\sum_{i=1}^{N}H_{i,t}^{2}\right\rVert_{p}dt.
\]
 Since the first process on the right-hand side of the inequality is a local martingale, there exists a sequence of stopping times $(T_{m})_{m\ge 1}$ such that, for all $m$, this process stopped at $T_{m}$ is a martingale with zero expectation, and $T_{m} \uparrow \infty$. Applying expectation at time $t\wedge T_{m}$, we get
\[\mathbb{E}\lVert X_{t\wedge T_{m}}\rVert_{2p}^{2} \le (2p-1)\mathbb{E}\int_{0}^{t\wedge T_{m}}\left\lVert \sum_{i=1}^{N}H_{i,s}^{2}\right\rVert_{p}ds \le (2p-1)\mathbb{E}\int_{0}^{t}\left\lVert \sum_{i=1}^{N}H_{i,s}^{2}\right\rVert_{p} ds.\]
Letting $m$ tend to infinity and using Fatou's lemma yields the result.
\end{proof}

\begin{remark} If we consider $f = \lVert \cdot\rVert_{2p}^{2p}$, the same proof yields \[\left(\mathbb{E}\left\lVert X_{t} \right\rVert_{2p}^{2p}\right)^{1/p} \le (2p-1)\int_{0}^{t}\left(\mathbb{E}\left\lVert \sum_{i=1}^{N}H_{i,s}^{2}\right\rVert_{p}^{p}\right)^{1/p}ds.\] Nonetheless, the issue here is the placement of the power $p$. \end{remark}

In fact, the assumption that the matrices are symmetric is not essential. It can be removed by a standard trick, see e.g. \cite{Tro15}.  We denote $\mathcal{M}_{n_{1}\times n_{2}}$ the space of rectangular matrices of $n_{1}$ raws and $n_{2}$ columns.  
For $A \in \mathcal{M}_{n_{1}\times n_{2}}$, we apply the symmetric case to the Hermitian dilatation, namely the matrix $\mathcal{H}(A)$ given by \begin{equation}\label{matrix} \mathcal{H}(A) = \begin{pmatrix} 0 & A \\ A^{*} & 0\end{pmatrix}.\end{equation} 
Recall that for $A \in \mathcal{M}_{n_{1}\times n_{2}}$, the Schatten $p-$ norm of $A$ is defined as $\lVert A\rVert_{p}^{p} = \sum_{i=1}^{\min\{n_{1   },n_{2}\}} \lvert\sigma_{i}(A)\rvert^{p}$ where $\sigma_{1}(A),\dots, \sigma_{\min\{n_{1},n_{2}\}}(A)$ are the singular values of $A$. Moreover, $A^{*}$ denotes the conjugate transpose of $A.$

\begin{theorem}\label{noncarre} Let $p \ge 1$, and $(X_{t})$ be a process of the form $(\ref{eq_Xt})$ but taking values in $\mathcal{M}_{n_{1}\times n_{2}}$. Then 
\[\mathbb{E}\lVert X_{t}\rVert_{2p}^{2} \le 2^{-1/p}\sqrt{2p-1}\int_{0}^{t}\mathbb{E}\left(\left\lVert\sum_{i=1}^{N}H_{i,s}H_{i,s}^{*}\right\rVert_{p}^{p}+ \left\lVert\sum_{i=1}^{N}H_{i,s}^{*}H_{i,s}\right\rVert_{p}^{p}\right)^{1/p}ds\]
\end{theorem}

\begin{proof}
We observe that $\mathcal{H}(X_{t}) = \int_{0}^{t}\sum_{i=1}^{N}\mathcal{H}(H_{i,s})dB_{s}^{i}$. Applying Theorem \ref{th_schatten} to the self-adjoint matrix $\mathcal{H}(X_{t})$, we obtain
\begin{equation}\label{a}\mathbb{E}\lVert\mathcal{H}(X_{t})\rVert^{2}_{2p} \le (2p-1)\int_{0}^{t}\mathbb{E}\left\lVert\sum_{i=1}^{N}\mathcal{H}(H_{i,s})^{2}\right\rVert_{p}ds.\end{equation}
But, for every matrix $A$,
\begin{equation}\label{carre}\mathcal{H}(A)^{2} = \begin{pmatrix} AA^{*}&0\\0&A^{*}A\end{pmatrix}.\end{equation}
Thus, \begin{equation}\left\lVert \sum_{i=1}^{N}\mathcal{H}(H_{i,s})^{2}\right\rVert_{p} =\left(\left\lVert\sum_{i=1}^{N}H_{i,s}H_{i,s}^{*}\right\rVert_{p}^{p}+\left\lVert\sum_{i=1}^{N}H_{i,s}^{*}H_{i,s}\right\rVert_{p}^{p}\right)^{1/p}.\label{b}\end{equation}
On the other hand, note that 
\begin{equation}\label{c}\lVert\mathcal{H}(X_{t})\rVert_{2p}^{p} = \left(\lVert X_{t}^{*}X_{t}\rVert_{p}^{p} + \lVert X_{t}X_{t}^{*}\rVert_{p}^{p}\right)^{1/p} = 2^{1/p}\lVert X_{t}^{*}X_{t}\rVert_{p} = 2^{1/p}\lVert X_{t}\rVert_{2p}^{2}.\end{equation}
By combining (\ref{a}), (\ref{b}) and (\ref{c}), we obtain the desired result.
\end{proof}

This result is consistent with the non-symmetric version of matrix Khintchine inequality, see \cite{Lus91}.

\section{The matrix Freedman inequality}

In this section, we prove Theorem \ref{freedmann}. To proceed, we first introduce the following lemma:

\begin{lemma} \label{lemma2}
Let $f : M \mapsto \Tr e^{M}$ where $M$ is taking values in the space of symmetric matrices. Then, for all symmetric matrices $M, H$; \[ \langle \nabla^{2}f(M)H;H\rangle \le \langle \nabla f(M); H^{2}\rangle = \Tr(e^{M}H^{2})\] where $\nabla^{2}f(M)$ stands for the Hessian matrix of $f$ at $M$.
\end{lemma}   

The proof of Lemma \ref{lemma2} can be found in Section 7 of \cite{Kla24}. The key lemma is as follows:
\begin{lemma}\label{lemma3} Let $(X_{t})_{t\ge 0}$ be a local martingale of the form $(\ref{eq_Xt})$. Then, the process $$\left(\Tr(e^{X_{t}-\frac{1}{2}\langle X\rangle_{t}})\right)_{t\ge 0}$$ is a supermartingale.
\end{lemma} 
\begin{proof} Consider the process $Y_{t} = X_{t} - \frac{1}{2}\langle X \rangle_{t}$ and let $f : A \rightarrow \Tr(e^{A})$ as defined in the previous lemma. By Itô's lemma, and Lemma \ref{lemma2}, we have  
\begin{align*} df(Y_{t}) &= \sum_{i=1}^{N}\left[\langle \nabla f(Y_{t});H_{i,t}\rangle dB_{t}^{i} - \frac{1}{2}\langle \nabla f(Y_{t});H_{i,t}^{2}\rangle dt\right] + \frac{1}{2}\sum_{i=1}^{N}\langle\nabla^{2}f(Y_{t})H_{i,t};H_{i,t}\rangle dt\\
 &\le \Tr\left(e^{Y_{t}}\left(\sum_{i=1}^{N}H_{i,t}dB_{t}^{i} - \frac{1}{2}\sum_{i=1}^{N}H_{i,t}^{2}dt\right)\right) + \frac{1}{2}\Tr\left(e^{Y_{t}}\sum_{i=1}^{N}H_{i,t}^2 dt\right)\\
&= \sum_{i=1}^{N}\Tr\left(e^{Y_{t}}H_{i,t}\right)dB_{t}^{i}.
\end{align*}
Therefore the nonnegative process $\left(\Tr\left(e^{Y_{t}}\right)\right)_{t\ge 0}$ is the sum of a local martingale, and a decreasing adapted process. Fatou's lemma ensures that it is a supermartingale.
\end{proof}

\begin{corollary}\label{allow} Let $(X_{t})$ be a stochastic process as defined in $(\ref{eq_Xt})$. Let $\sigma, u \in \mathbb{R}$ and $t>0$. If \[\mathbb{P}\left(\left\lVert \langle X \rangle_{t} \right\rVert \le \sigma^{2}\right) = 1.\] Then, \begin{equation}\label{amodifier}\mathbb{P}\left(\lambda_{\text{max}}(X_{t}) \ge u\right) \le ne^{-\frac{u^{2}}{2\sigma^{2}}}.\end{equation}
\end{corollary}

\begin{proof}
We almost surely have $\langle X \rangle_{t} \preccurlyeq \sigma^{2}I_{n}$, where $A\preccurlyeq B$ means that the matrix $B-A$ is a positive semi-definite symmetric matrix. Fix $\beta \ge 0.$ Then 
\[
	\Tr\left(e^{\beta X_{t} - \frac{\beta^{2}}{2}\langle X \rangle_{t}}\right) \ge \Tr\left(e^{\beta X_{t} - \beta^{2}\frac{\sigma^{2}}{2}I_{n}}\right) = \Tr(e^{\beta X_{t}})e^{-\frac{\beta^{2}\sigma^{2}}{2}} \ge e^{\beta\lambda_{\text{max}}(X_{t}) - \frac{\beta^{2}\sigma^{2}}{2}}.\]
Combining with Lemma \ref{lemma3}, we obtain
\[\mathbb{E}[e^{\beta\lambda_{\text{max}}(X_{t})}]\le e^{\frac{\beta^{2}\sigma^{2}}{2}}\mathbb{E}\Tr\left(e^{\beta X_{0}-\frac{\beta^{2}\langle X\rangle_{0}}{2}}\right) = ne^{\frac{\beta^{2}\sigma^{2}}{2}}.\]
We conclude by Chernoff's inequality.
\end{proof}

\begin{proof}[Proof of Theorem \ref{freedmann}] Let $\epsilon > 0$ and define the stopping time \[\tau = \inf\{t > 0, \; \lambda_{\text{max}}(X_{t}) \ge u \; \; \text{or} \; \; \lVert\langle X\rangle_{t}\rVert \ge \sigma^{2} + \epsilon\}.\]
Let $X^{\tau}$ be the process $X$ stopped at time $\tau$, namely $X^{\tau} = X_{t \wedge \tau}$. 
We observe that $\langle X^{\tau} \rangle_{t} = \langle X \rangle_{t\wedge \tau}$. By definition of $\tau$ and by continuity of $t \rightarrow \langle X \rangle_{t}$, we have \[\lVert\langle X_{t}^{\tau}\rangle\rVert \le \sigma^{2}+\epsilon\] almost surely. Then by Corollary \ref{allow},
\begin{equation}\label{diffusion}\mathbb{P}\left(\lambda_{\text{max}}(X_{t}^{\tau}) \ge u\right) \le ne^{-\frac{u^{2}}{2(\sigma^{2}+\epsilon)}}.\end{equation} 
Now, if there exists $s \le t$ such that $\lambda_{\text{max}}(X_{s}) \ge u$ and $\lVert\langle X\rangle_{s}\rVert \le \sigma^{2}$, then $\tau \le s \le t$ and
\[\lVert\langle X\rangle_{\tau}\rVert \le \lVert\langle X\rangle_{s}\rVert \le \sigma^{2} < \sigma^{2}+\epsilon.\]
Then, by definition of $\tau$, we must have \[\lambda_{\text{max}}(X_{\tau}) = \lambda_{\text{max}}(X_{t\wedge\tau}) \ge u.\]
Combining with (\ref{diffusion}), we get \begin{align*}\mathbb{P}(\exists   s \le t,\; \lambda_{\text{max}}(X_{s}) \ge u \;  \text{and}  \; \lVert\langle X \rangle_{s}\rVert \le \sigma^{2}) 
 &\le \mathbb{P}(\lambda_{\text{max}}(X_{t}^{\tau}) \ge u) \\  &\le  ne^{-\frac{u^{2}}{2(\sigma+\epsilon)^{2}}}.\end{align*}
We obtain Theorem \ref{freedmann} by letting $\epsilon$ tend to 0 and $t$ to $\infty$ and using monotone convergence. 
\end{proof}

\begin{remark}
    We do not insist that $(\mathcal{F}_{t})$ is the natural filtration of the Brownian motion, $(\mathcal{F}_{t})$ could be richer. The proof of Corollary \ref{allow} actually shows that $$\mathbb{P}(\lambda_{\text{max}}(X_{t})\ge u \lvert\mathcal{F}_{0}) \le e^{-\frac{u^{2}}{2\sigma^{2}}}$$ almost surely, which is slightly stronger than $(\ref{amodifier})$. In the same way, under the assumptions of Theorem \ref{freedmann}, the conclusion of Theorem \ref{freedmann} can be replaced by \begin{equation}\label{enfin}
        \mathbb{P}(\exists t > 0, \; \lambda_{\text{max}}(X_{t}) \ge u \;\text{and}\; \lVert\langle X\rangle_{t}\rVert \le \sigma^{2} \mid \mathcal{F}_{0}) \le ne^{-\frac{u^{2}}{2\sigma^{2}}}\quad a.s.
    \end{equation} 
\end{remark}

We now prove Theorem \ref{good}.

\begin{proof}[Proof of Theorem \ref{good}] Let $\tau := \inf\{t>0,\; \lambda_{\text{max}}(X_{t})>u\}$. Since $(X_{t})$ takes values in symmetric matrices, by Weyl's inequality (see, e.g., \cite[Chapter 3]{Bha97}),then we have\begin{align}\lambda_{\text{max}}(X_{\tau +t}) &\le \lambda_{\text{max}}(X_{\tau}) + \lambda_{\text{max}}(X_{\tau +t}-X_{\tau})\nonumber\\ \label{numero1}&= u+\lambda_{\text{max}}(X_{\tau +t}-X_{\tau}).\end{align} Moreover, since the matrix $\langle X\rangle_{\tau}$ is positive, we also have \begin{equation}\label{numero2}
    \lVert \langle X\rangle_{\tau +t}\rVert \ge \lVert \langle X\rangle_{\tau +t}-\langle X\rangle_{\tau}\rVert.
\end{equation} Set $\widetilde{X}_{t} = X_{\tau +t}-X_{\tau}$, $\widetilde{H_{i,t}} = H_{i,\tau +t}$, $\widetilde{\mathcal{F}_{t}} = \mathcal{F}_{\tau +t}$ and $\widetilde{B}_{t} = B_{\tau +t}-B_{\tau}$. By the strong Markov property of Brownian motion, $(\widetilde{B}_{t})$ is a Brownian motion with respect to $(\widetilde{\mathcal{F}}_{t})$. Moreover $(\widetilde{H}_{i,t})$ is progressively measurable with respect to $(\widetilde{\mathcal{F}}_{t})$ and $$\widetilde{X}_{t} = \int_{\tau}^{\tau +t}\sum_{i=1}^{N}H_{i,s}dB_{s}^{i}  =\int_{0}^{t}\sum_{i=1}^{N}\widetilde{H}_{i,s}d\widetilde{B}_{s}^{i}.$$ In particular, $$\langle\widetilde{X}\rangle_{t} = \int_{0}^{t}\sum_{i=1}^{N}\widetilde{H}_{i,s}^{2}ds = \langle X\rangle_{\tau +t}-\langle X\rangle_{\tau}.$$ From $(\ref{numero1})$ and $(\ref{numero2})$, we get 
\begin{align*}
    &\mathbb{P}(\exists t>0,\lambda_{\text{max}}(X_{t})\ge 2u \; \text{and}\; \lVert\langle X\rangle_{t}\rVert \le \sigma^{2}) \\ &= \mathbb{P}(\tau < \infty, \;\exists t>0; \lambda_{\text{max}}(X_{\tau+t})\ge 2u \;\text{and}\;\lVert\langle X\rangle_{\tau +t}\rVert \le \sigma^{2}) \\ &\le \mathbb{P}(\tau < \infty, \; \exists t>0, \lambda_{\text{max}}(\widetilde{X}_{t}) \ge u \; \text{and}\; \lVert\langle \widetilde{X}\rangle_{t}\rVert \le \sigma^{2}) \\ &= \mathbb{E}\left[\mathbb{1}_{\{\tau < \infty\}}\mathbb{P}(\exists t > 0, \lambda_{\text{max}}(\widetilde{X}_{t})\ge u \;\text{and}\;\lVert\langle \widetilde{X}\rangle_{t}\rVert \le \sigma^{2} \mid \mathcal{F}_{\tau})\right]
\end{align*}
But $\mathcal{F}_{\tau} = \tilde{\mathcal{F}}_{0}$ and from (\ref{enfin}), we get $$\mathbb{P}(\exists t>0,\lambda_{\text{max}}(X_{t})\ge 2u \; \text{and}\; \lVert\langle X\rangle_{t}\rVert \le \sigma^{2}) \le ne^{-\frac{u^{2}}{2\sigma^{2}}}\mathbb{P}(\tau < \infty)$$
which allows us to conclude by the definition of $\tau$.
\end{proof}

As before, we can extend Theorem \ref{good} to non-symmetric and rectangular matrices. We get the following :
\begin{theorem} 
Consider $(X_{t})$ of the form $(\ref{eq_Xt})$ but taking values in $\mathcal{M}_{n_{1}\times n_{2}}$ the space of rectangular matrices of $n_{1}$ rows and $n_{2}$ columns. Then,
\[\mathbb{P}\left(\exists t> 0, \; \lVert X_{t}\rVert \ge 2u \; \; \text{and}\;\; \Lambda_{t} \le \sigma^{2}\right) \le (n_{1}+n_{2})e^{-u^{2}/(2\sigma^{2})}\mathbb{P}(\exists t>0,\; \lVert X_{t}\rVert \ge u)\]
where \begin{equation}\label{lambda}\Lambda_{t} = \max\left\{\left\lVert\int_{0}^{t}\sum_{i=1}^{N}H_{i,s}H_{i,s}^{*}ds\right\rVert ;\left\lVert\int_{0}^{t}\sum_{i=1}^{N}H_{i,s}^{*}H_{i,s}ds\right\rVert\right\}.\end{equation}
\end{theorem}   

As for Theorem \ref{noncarre}, the proof consist in applying the symmetric case to $\mathcal{H}(X_{t})$ where $\mathcal{H}$ is defined by $(\ref{matrix}),$ and using the relation \begin{equation}\label{utile}\lambda_{\text{max}}(\mathcal{H}(A)) = \lVert \mathcal{H}(A) \rVert = \lVert A \rVert.\end{equation} The details are left to the reader.

\section{A Burkholder-Davis-Gundy type inequality for the spectral norm}

In this section, we prove Theorem \ref{thm_main}. The proof relies on a classical trick to derive inequalities of this type from a good-$\lambda$ inequality, as presented for instance in \cite{Ban98} or \cite[Chapter IV]{Rev99}. We include the proof for completeness.

\begin{proof}[Proof of Theorem \ref{thm_main}]
    Let $t>0$ and $X_{t}^{*} = \sup_{0\le s\le t}\lVert   X_{s}\rVert$. Let $\delta > 0$. Using $(\ref{carre})$ and the fact that $(H_{i,t})$ take values in the set of symmetric matrices, we have \begin{equation}\label{rmq}\lVert\langle \mathcal{H}(X)\rangle_{t}\rVert = \lVert\langle X\rangle_{t}\rVert.\end{equation} Then, combining Theorem \ref{good} to the symmetric matrix $\mathcal{H}(X_{s}^{t}) = \mathcal{H}(X_{s\wedge t})$ with (\ref{utile}) and $(\ref{rmq})$, we get 
    \begin{equation}\label{nvxfreedmann}\mathbb{P}(\exists s \in (0,t],\; \lVert X_{s}\rVert\ge 2u \;\; \text{and}\;\;\lVert\langle X\rangle_{s}\rVert\le\sigma^{2}) \le 2ne^{-\frac{u^{2}}{2\sigma^{2}}}\mathbb{P}(\exists s\in (0,t],\; \lVert X_{s}\rVert \ge u).\end{equation}
    Moreover, since the process $t \mapsto \lVert\langle X\rangle_{t}\rVert$ is nondecreasing, \begin{equation}\label{utilee}
        \mathbb{P}(X_{t}^{*} > 2u \;\;\text{and}\;\;\lVert\langle X\rangle_{t}\rVert^{1/2} \le \delta u) \le \mathbb{P}(\exists s \in (0,t],\;\lVert X_{s}\rVert > 2u \;\;\text{and}\;\;\lVert\langle X\rangle_{s}\rVert^{1/2} \le \delta u).
    \end{equation}

    \noindent{}By Theorem \ref{good}, $(\ref{utilee})$ and $(\ref{nvxfreedmann})$, \begin{align*}\mathbb{E}( X_{t}^{*}/2)^{p} &= p\int_{0}^{\infty}u^{p-1}\mathbb{P}(X_{t}^{*} >2u)du
    \\ &\le p\int_{0}^{\infty}u^{p-1}\mathbb{P}(X_{t}^{*}>2u, \lVert\langle X\rangle_{t}\rVert^{1/2} \le \delta u)du + p\int_{0}^{\infty}u^{p-1}\mathbb{P}(\lVert\langle X\rangle_{t}\rVert^{1/2} > \delta u)du\\ 
    &\le 2npe^{-1/2\delta^2}\int_{0}^{\infty}u^{p-1}\mathbb{P}(X_{t}^{*}>u)du + \frac{1}{\delta^{p}}\mathbb{E}[\lVert \langle X\rangle_{t}\rVert^{p/2}] \\ &= 2ne^{-1/2\delta^2}\mathbb{E}[(X_{t}^{*})^{p}] + \frac{1}{\delta^{p}}\mathbb{E}[\lVert \langle X\rangle_{t}\rVert^{p/2}].\end{align*}
Then, \begin{equation}\label{final}\mathbb{E}[(X_{t}^{*})^{p}]\left(\frac{1}{2^{p}}-2ne^{-1/2\delta^2}\right) \le \frac{1}{\delta^{p}}\mathbb{E}[\lVert \langle X\rangle_{t}\rVert^{p/2}].\end{equation}
 Taking $\delta = \frac{1}{\sqrt{2}\sqrt{(p+2)\log 2 + \log n}}$ yields $\frac{1}{2^{p}}-ne^{-1/2\delta^{2}} = 2^{-(p+1)}$. Injecting this into (\ref{final}), we obtain \begin{align*}
        \mathbb{E}(X_{t}^{*})^{p} &\le 2^{p+1}(\sqrt{2})^{p}((p+2)\log 2+\log(n))^{p/2}\mathbb{E}\lVert \langle X\rangle_{t}\rVert^{p/2}. 
    \end{align*} 
Then, \begin{align*} (\mathbb{E}(X_{t}^{*})^{p})^{1/p} &\le 2^{1+1/p}\sqrt{2}\sqrt{(p+2)\log 2 + \log n}\;(\mathbb{E}\lVert \langle X\rangle_{t}\rVert^{p/2})^{1/p}\\ & \le 12\sqrt{2\log(2)}\sqrt{p+\log n}\;(\mathbb{E}\lVert \langle X\rangle_{t}\rVert^{p/2})^{1/p}.
\qedhere \end{align*}
\end{proof}

As in Theorem \ref{th_schatten}, the symmetry assumption is not essential. Indeed, using the same notation from section 1, and applying Theorem \ref{thm_main} to $\mathcal{H}(X_{t})$. We obtain the following theorem. 
\begin{theorem}
Assuming that for all integers $i$, $(H_{i,s})_{s}$ are progressively measurable processes taking values in the space of rectangular matrices with $n_{1}$ rows and $n_{2}$ columns. Then, 
\[\left(\mathbb{E}\sup_{0\le s\le t}\lVert X_{s}\rVert^{p}\right)^{1/p} \le C\sqrt{p+2\log(n_{1}+n_{2})}\left(\mathbb{E}\;\Lambda_{t}^{p/2}\right)^{1/p}\] where $\Lambda_{t}$ is defined by $(\ref{lambda}).$ 
\end{theorem}
The proof is left to the reader, as it follows similar arguments to those of Theorem \ref{noncarre}.

\end{document}